\documentclass{article}

\usepackage{amsmath,amssymb,amsfonts,graphicx,amsthm,nicefrac}
\usepackage{hyperref}
\usepackage{natbib}
\usepackage[english]{babel}
\usepackage{times}
\usepackage[T1]{fontenc}

\newcommand{\ignore}[1]{}

\newtheorem{theorem}{Theorem}

\newtheorem{lemma}{Lemma}

\theoremstyle{definition}
\newtheorem{definition}{Definition}

\theoremstyle{remark}

\makeatletter
\newtheorem*{rep@theorem}{\rep@title}
\newcommand{\newreptheorem}[2]{%
\newenvironment{rep#1}[1]{%
 \def\rep@title{#2 \ref{##1}}%
 \begin{rep@theorem}}%
 {\end{rep@theorem}}}
\makeatother

\newreptheorem{theorem}{Theorem}
\newreptheorem{lemma}{Lemma}
\newreptheorem{corollary}{Corollary}

\newcommand{\norm}[1]{\left\lVert{#1}\right\rVert}
\newcommand{\EE}[1]{\mathbb{E}\left[{#1}\right]}

\renewcommand{\O}[1]{\mathbf{O}\left({#1}\right)}
\newcommand{\Th}[1]{\mathbf{\Theta}\left({#1}\right)}

\def\R{\mathbb{R}}
\def\Z{\mathbb{Z}}

  \def\bhat{\hat{\beta}}
  \def\tb{\tilde{\beta}}
  \def\b{\beta}
  \def\hatbl{\hat{\beta}_{\lambda}}
  \def\hatbb{\hat{\beta}^B}
  \def\l{\lambda}
  \def\s{\sigma}
  \def\S{\Sigma}
  \def\smax{\sigma_{\text{max}}}
  
  \def\e{\epsilon}
  \def\bstar{\b^*}
  
  \def\U{\mathcal{U}}

\usepackage[left=3cm,right=3cm,top=3cm,bottom=3cm]{geometry}
  
\title{Fast-rate and optimistic-rate error bounds \\ for $\ell_1$-regularized regression}
\author{Rina Foygel and Nathan Srebro}
\begin{document}

\maketitle

\begin{abstract}
  We consider the prediction error of linear regression with $\ell_1$
  regularization when the number of covariates $p$ is large relative
  to the sample size $n$.  When the model is $k$-sparse and
  well-specified, and restricted isometry or similar conditions hold,
  the excess squared-error in prediction can be bounded on the order
  of $\frac{\sigma^2 k\log(p)}{n}$, where $\sigma^2$ is the noise
  variance.  Although these conditions are close to necessary for
  accurate {\em recovery} of the true coefficient vector, it is
  possible to guarantee good predictive accuracy under much milder
  conditions, avoiding the restricted isometry condition, but only
  ensuring an excess error bound of order
  $\frac{k\log(p)}{n}+\sigma\sqrt{\frac{k\log(p)}{n}}$.  Here we show
  that this is indeed the best bound possible (up to logarithmic
  factors) without introducing stronger assumptions similar to
  restricted isometry.
\end{abstract}

\section{Introduction}
We consider a random design linear regression problem with $p$
covariates:
\begin{equation*}
  y = x^T \beta^* + z
\end{equation*}
where $x \in \R^p$ are random covariates with covariance matrix
$\Sigma$, $z$ is random noise with $\EE{z^2}=\sigma^2$, and $\beta^*
\in \R^p$ are the regression coefficients.  For simplicity we take the
response to be normalized, $\EE{y^2}=1$ (otherwise all results scale
accordingly).

We consider the problem of minimizing the prediction error
\begin{equation*}
\EE{\left(y-x^T\beta\right)^2}
\end{equation*}
based on an i.i.d.~sample $\left(x^{(1)},y^{(1)}\right)\ldots,\left(x^{(n)},y^{(n)}\right)$ using
$\ell_1$-regularized regression:
\begin{equation*}
  \hatbb\doteq \arg\min_{\|\b\|_1\leq B}\sum_i \left(y^{(i)}-x^{(i)}{}^T\beta\right)^2\;.
\end{equation*}
Note that up to some unknown and data-dependent correspondence between
$B$ and $\lambda$, this is the same as
\begin{equation*}
  \hatbl\doteq \arg\min_{\b}\sum_i \left(y^{(i)}-x^{(i)}{}^T\beta\right)^2 +
  \lambda \norm{\b}_1\;,
\end{equation*}
also known as Lasso regression \citep{Lasso}.

Suppose that the covariates are $1$-bounded, and that $\max_i \left|y^{(i)}\right|\leq \O{\log(np)}$ (for instance, this is true with high probability in the Gaussian setting). Then, by \citet{SST}, with high probability over the sample, for any fixed $\beta^*$ with $\|\beta^*\|_1\leq B$, excess squared-error under $\ell_1$-regularized regression is bounded as
\begin{equation}\label{eq:OptimisticRate_ell1}\EE{\left(y-x^T\hatbb\right)^2}- \EE{\left(y-x^T{\beta}^*\right)^2} = \O{\frac{(1+B)^2\log(p)}{n/\log^3(n)}+\sqrt{\frac{(1+B)^2\log(p)}{n/\log^3(n)}\cdot \EE{\left(y-x^T{\beta}^*\right)^2}}}\;.\end{equation}

This result does not require any conditions on the correlation between the covariates, or on the nature of the ``noise'' $y-x^T\beta^*$, aside from the mild bound on $\max_i \left|y^{(i)}\right|$. In particular, this noise is not required to be independent from $x$. We believe also that this result would hold for subgaussian $x$'s (rather than our current stronger assumption that the $x$'s are $1$-bounded).

We can apply this result to the sparse regression setting, with some mild additional assumptions. Suppose that we are interested in comparing to a sparse predictor on an unknown support $J^*\subset [p]$, with $|J^*|\leq k$. We now place a lower-bound eigenvalue assumption on this support $J^*$ only:
\begin{equation}\label{eq:lambda1}\lambda_{\mathrm{min}}\left(\EE{x_{J^*}x_{J^*}^T}\right)\geq \lambda_1>0\;,\end{equation}
where $x_{J^*}=\left(x_j:j\in J^*\right)$ is the random vector consisting of those covariates $x_j$ for which $\bstar_j$ is nonzero.
This assumption is strictly weaker than the restricted isometry property (RIP) conditions in the compressed sensing literature, which require an upper-bound assumption as well, and require the eigenvalue bounds to hold for all sets $J\subset [p]$ of bounded size, in addition to the true support $J^*$.

We fix the scale of the problem by assuming $\EE{y^2}=1$. Now consider a predictor $\beta^*$ with support in $S^*$, which is better than the zero predictor --- that is, $\EE{(y-x^T\beta^*)^2}\leq \EE{(y-x^T\mathbf{0}_p)^2}=\EE{y^2}$=1. We now show that $\|\beta^*\|_1=\O{\sqrt{k\l_1^{-1}}}$. We first bound $\|\beta^*\|^2_2$, by observing that
\begin{align*}
&\|\beta^*\|^2_2\cdot \lambda_1\leq (\beta^*)^T \EE{xx^T}\beta^*
=\EE{\left(x^T\beta^*\right)^2}
=\EE{\left(y-x^T\beta^*\right)^2}-2\EE{y\cdot \left(y-x^T\beta^*\right)}+\EE{y^2}\\
&\leq 2\EE{\left(y-x^T\beta^*\right)^2}+2\EE{y^2}\leq 4\;.
\end{align*}
We then have
$$\|\beta^*\|_1\leq \sqrt{k}\|\beta^*\|_2\leq \sqrt{k\cdot 4\lambda_1^{-1}}=\O{\sqrt{k\l_1^{-1}}}\;.$$

Therefore, with high probability,
\begin{equation}\label{eq:OptimisticRate}
\EE{\left(y-x^T\hatbb\right)^2}- \EE{\left(y-x^T{\beta}^*\right)^2} = \O{\frac{k\log(p)}{\l_1 n/\log^3(n)}+\sqrt{\frac{k\log(p)}{\l_1 n/\log^3(n)}\cdot \EE{\left(y-x^T{\beta}^*\right)^2}}}\;,\end{equation}
under the assumption that the $x$'s are $1$-bounded and $\max_i \left|y^{(i)}\right|\leq \mathbf{O}(np)$. Therefore, to guarantee a bound of $\epsilon$ on the excess prediction error, the required sample complexity is
\begin{equation}\label{eq:OptimisticRate_sc}
n=\Th{\frac{k\log(p)}{\l_1\e}\cdot\frac{\s^2+\e}{\e}\cdot\log^3\left(\nicefrac{k}{\l_1\e}\right)}\;,
\end{equation}
where $\s^2=\EE{\left(y-x^T\bstar\right)^2}$ is the magnitude of the noise. This sample complexity follows an ``optimistic rate'': in the noisy setting, if we would like to ensure a bound $\e$ on excess error which is small relative to $\sigma^2$, then the required sample complexity is then $n= \Th{\e^{-2}}$, but on the the other hand, in the noiseless setting (i.e. when $y=x^T\beta^*$), or if the bound on excess error $\e$ is not much smaller than $\sigma^2$, then we require only $n= \Th{\e^{-1}}$. We emphasize that this result does not assume that the linear model is a true model or require independent noise.

In contrast, results on sparse vector recovery from the compressed sensing framework \citep{CandesTao,BRT,Koltchinskii, Cai} provide stronger guarantees in a similar setting, using either $\ell_1$-regularized regression or the Dantzig selector, given by
$$\bhat_{\lambda}^{DS}=\arg\min \max_i \left|y^{(i)}-x^{(i)}{}^T\beta\right|+\lambda\|\beta\|_1\;.$$
These stronger results require several additional specialized assumptions, including the requirement that the noise must be independent from the signal. Existing results are stated either in the deterministic or random covariates setting, but can in general be translated to a random Gaussian setting. We restrict our attention to $\ell_1$-regularized regression when the covariates are i.i.d.~multivariate Gaussian with zero mean: $x^{(i)}\stackrel{i.i.d.}{\sim}N(0,\Sigma)$. We now summarize this setting (with some simplifications), and compare it to the optimistic-rate results discussed above.
\begin{itemize}
\item {\bf Well-specified model with independent subgaussian noise:} Response $y^{(i)}$ is given by $y^{(i)}=x^{(i)}{}^T\beta^*+\sigma z^{(i)}$ for a true predictor $\beta^*$ satisfying $\|\beta^*\|_1\leq B$, and $z^{(i)}$ is a subgaussian or subexponential noise term with unit variance, and is independent from $x^{(i)}$.\\
The main additional requirement here is that noise $z$ is independent of $x$.  This in particular implies that $\bstar$ is the optimal regressor. Note that in order to obtain the optimistic-rate guarantee~(\ref{eq:OptimisticRate}), no such assumption is necessary, and $\bstar$ can be a non-optimal regressor chosen for its sparsity or eigenvalue properties.
\item {\bf Sparsity:} $\beta^*$ is $k$-sparse, meaning that it has (at most) $k$ non-zero entries.\\
To obtain the optimistic-rate guarantee as stated originally in~(\ref{eq:OptimisticRate_ell1}), we can relax this requirement and only assume that $\bstar$ has low $\ell_1$-norm.
\item {\bf Restricted eigenvalues:} There exists a $\kappa\doteq\kappa(k,3)>0$, such that for any $J\subset [p]$ with $|J|\leq k$, for any nonzero $\beta\in\R^p$ with $\|\beta_{\overline{J}}\|_1\leq 3\|\beta_J\|_1$,
\begin{equation}\label{eq:RE_kappa}\beta^T\Sigma\beta\geq \kappa\|\beta_J\|^2_2\;.\end{equation}
This restricted eigenvalue condition is implied by a stronger condition:\\
{\bf Restricted isometry:} Suppose that $\delta_{2k}+3\theta_{k,2k}<1$, where
$v^T\Sigma v\in \left(1\pm \delta_{2k}\right)\|v\|^2_2$ for all $2k$-sparse vectors $v$, and
$\left|v^T\Sigma w\right|\leq \theta_{k,2k}\|v\|_2\|w\|_2$ for all $k$-sparse $v$ and $2k$-sparse $w$ with disjoint supports. 
Then $\kappa\doteq \sqrt{1-\delta_{2k}}\left(1-\frac{3\theta_{k,2k}}{1-\delta_{2k}}\right)$ satisfies the restricted eigenvalue condition above.\\
To obtain the optimistic-rate guarantee~(\ref{eq:OptimisticRate}) under the sparsity assumption, we required an eigenvalue condition~(\ref{eq:lambda1}) on $\Sigma_{\mathrm{Support}\left(\bstar\right)}$ only, which is strictly weaker than the restricted eigenvalue and restricted isometry assumptions.
\end{itemize}

Under these assumptions, with $\kappa$ defined as in~(\ref{eq:RE_kappa}), the following guarantees hold with high probability, by Theorem 7.2 of \citet{BRT}:
\begin{align}
\notag&\text{Sparse and accurate estimation of $\beta^*$: }
\left\|\hatbb-\beta^*\right\|_1 =\O{\frac{\sigma k}{\kappa^2}\cdot \sqrt{\frac{\log(p)}{n}}}\;,\text{ and }\|\hatbb\|_0=\O{k}\;.\\
\label{eq:FastRate}&\text{Bounded excess prediction error: }\EE{\left(y-x^T\hatbb\right)^2}= \EE{\left(y-x^T\beta^*\right)^2}+\O{\frac{\sigma^2 k\log(p)}{\kappa^2n}}\;.
\end{align}

This corresponds to a sample complexity of
\begin{equation}\label{eq:FastRate_sc}
n=\Th{\frac{\s^2k\log(p)}{\kappa^2\e}}\;,\end{equation}
to ensure an excess error bound of $\e$. It is crucial to note that the error bound (and the sample complexity) scales with the magnitude of the noise, $\sigma^2$, rather than to the (unit) magnitude of the signal. In particular, in a noiseless setting, the results above guarantee a zero-error reconstruction of $\beta^*$, in contrast to the ``optimistic rate'' result~(\ref{eq:OptimisticRate}) where no such guarantee is given. Furthermore, in the noisy setting, the compressed sensing guarantees give a ``fast rate'' result, since the sample complexity scales with $\frac{1}{\e}$ rather than with $\frac{1}{\e^2}$.

In this compressed sensing framework, the guarantees on predictive error follow from a stronger guarantee on the accurate recovery of $\beta^*$, and in particular, the recovery of the true support of $\beta^*$. In order for this to be possible, it is of course necessary to be able to distinguish between pairs or small sets of covariates.  In particular, some sort of restricted isometry assumption is clearly necessary for bounding error in recovering $\beta^*$ (otherwise, the ``best'' $\beta^*$ might not be unique). However, if the goal is merely low error in prediction --- that is, we would like accuracy in calculating $x^T\beta^*$, rather than in recovering $\beta^*$ --- then perhaps this assumption could be weakened. For example, if a covariate is duplicated in the model, then it will not be possible to distinguish between the two when attempting to recover the true support; however, adding duplicated covariates to a model will have no effect on the problem of prediction.

 More generally, we are interested in whether the properties that are necessary for the (unique) recovery of $\beta^*$, are also necessary to obtain strong bounds on excess prediction error, and in the role of the assumptions that separate the ``optimistic rate'', unit-scale error bounds of \citet{SST} from the ``fast rate'' error bounds in the compressed sensing literature, which scale with the magnitude of the noise. Below, we show that, if we remove either the sparsity assumption (while still assuming that $\bstar$ has low $\ell_1$-norm) or the restricted isometry assumption from the compressed sensing framework described above, then up to logarithmic factors, the ``optimistic rate'' bound on excess prediction error, given in~(\ref{eq:OptimisticRate}), is the best possible bound. In particular, this implies that, even in the noiseless setting, we cannot achieve zero error in prediction, without stronger assumptions.

\section{Results}

First, we ask whether we can relax the assumption of a sparse true coefficient vector to an assumption on its $\ell_1$-norm, but still guarantee a fast-rate bound on excess error. Specifically, we consider the question of bounding excess prediction error, in the well-specified Gaussian setting where the restricted eigenvalue assumption holds, assuming only an $\ell_1$-norm bound on the true vector of coefficients.

Our first result shows that, up to logarithmic factors, the optimistic-rate error bound~(\ref{eq:OptimisticRate}) is the best possible rate under these conditions. For simplicity, we will consider the case of completely independent covariates, $x\sim N(0,\mathbf{I}_p)$. In particular, this ensures that the restricted eigenvalue assumption is satisfied. To place the problem on a unit scale (or rather, to bound the scale away from zero and away from infinity), we consider only true coefficient vectors $\beta^*$ satisfying
$$\frac{1}{2}\leq \EE{\left(x^T\bstar\right)^2}^{\nicefrac{1}{2}}=\|\bstar\|_2\leq \|\bstar\|_1\leq 1\;.$$

\begin{theorem}\label{sparsity_thm} Fix any $n\geq 30$, $p\geq 3n$, and $\sigma\geq 0$. Then there exists a $\beta^*\in\R^p$ with $\frac{1}{2}\leq\|\bstar\|_2\leq\|\beta^*\|_1\leq 1$, such that for any sample, for all $B\geq 0$,
\begin{equation*}
\EE{\left(y-x^T\hatbb\right)^2}\geq \sigma^2+ \frac{1}{32n\log^2(3n)}\;.
\end{equation*}
Additionally, if $100\leq \nicefrac{\sqrt{n}}{\s}\leq p$, then with probability at least $\frac{1}{2}$ over the sample, for all $B\geq 0$,
\begin{equation*}
\EE{\left(y-x^T\hatbb\right)^2}\geq \sigma^2+ \frac{\s}{102400\sqrt{n}\log^2\left(\max\left\{3n,\left\lceil{\nicefrac{\sqrt{n}}{\sigma}}\right\rceil\right\}\right)}\;.
\end{equation*}
Here $\hatbb=\arg\min_{\|\beta\|_1\leq B}\sum_i \left(y^{(i)}-x^{(i)}{}^T\beta\right)^2$, where $\left(x^{(i)},y^{(i)}\right)$ are i.i.d.~samples from the multivariate Gaussian distribution defined by drawing $x^{(i)}\sim N(0,\mathbf{I}_p)$ and $y^{(i)}\sim N\left(x^{(i)}{}^T\bstar,\s^2\right)$.  The expectations are taken over a new sample $(x,y)$ drawn from the same distribution, independently of the training set $\left\{\left(x^{(1)},y^{(1)}\right),\dots,\left(x^{(n)},y^{(n)}\right)\right\}$. (For each $B\geq 0$, if $\hatbb$ is not unique, then we show that the inequalities hold for some choice of $\hatbb$.)
\end{theorem}

Next, we ask whether the restricted eigenvalue (or restricted isometry) assumption is necessary for a fast-rate bound on excess error, in the well-specified Gaussian setting where the sparsity assumption holds. 

Our second result shows that, up to logarithmic factors, the optimistic-rate error bound~(\ref{eq:OptimisticRate}) is the best possible rate under these conditions. For simplicity, we restrict our attention to $2$-sparse true coefficient vectors. We also only consider covariance matrices $\Sigma$ such that $\Sigma_{J^*}=\mathbf{I}_{J^*}$, where $J^*=\mathrm{Support}\left(\bstar\right)$. That is, ensuring the restricted isometry property on the true support only, is not sufficient for a fast-rate bound on excess error.

To avoid issues of scaling, we restrict our attention to covariance matrices $\Sigma$ with $\|\Sigma\|_{sp}\leq 2$, and to true coefficient vectors $\beta^*$ satisfying
$$\frac{1}{2}\leq \EE{\left(x^T\bstar\right)^2}^{\nicefrac{1}{2}}=\sqrt{\bstar{}^T\Sigma\bstar}= \|\bstar\|_2\leq\|\bstar\|_1\leq 1\;,$$
where we make use of the fact that $\Sigma_{J^*}=\mathbf{I}_{J^*}$ to obtain the second equality.

\begin{theorem}\label{RIP_thm} Fix any $n\geq 30$, $p\geq 3n$, and $\sigma\geq 0$. Then there exists a $2$-sparse $\beta^*\in\R^p$ with $\frac{1}{2}\leq\|\bstar\|_2\leq\|\beta^*\|_1\leq 1$, and a positive semi-definite $\Sigma\in\R^{p\times p}$ with $\|\Sigma\|_{sp}\leq 2$ and $\Sigma_{\mathrm{Support}\left(\bstar\right)}=\mathbf{I}_{\mathrm{Support}\left(\bstar\right)}$, such that for any sample, for all $B\geq 0$,
\begin{equation*}
\EE{\left(y-x^T\hatbb\right)^2}\geq \sigma^2+ \frac{1}{288n\log^2(3n)}\;.
\end{equation*}
Additionally, if $100\leq \nicefrac{\sqrt{n}}{\s}\leq p-3$, then with probability at least $\frac{1}{2}$ over the sample, for all $B\geq 0$,
\begin{equation*}
\EE{\left(y-x^T\hatbb\right)^2}\geq \sigma^2+ \frac{\s}{409600\sqrt{n}\log^2\left(\max\left\{3n,\left\lceil{\nicefrac{\sqrt{n}}{\sigma}}\right\rceil\right\}\right)}\;.
\end{equation*}
Here $\hatbb=\arg\min_{\|\beta\|_1\leq B}\sum_i \left(y^{(i)}-x^{(i)}{}^T\beta\right)^2$, where $\left(x^{(i)},y^{(i)}\right)$ are i.i.d.~samples from the multivariate Gaussian distribution defined by drawing $x^{(i)}\sim N(0,\Sigma)$ and $y^{(i)}\sim N\left(x^{(i)}{}^T\bstar,\s^2\right)$. The expectations are taken over a new sample $(x,y)$ drawn from the same distribution, independently of the training set $\left\{\left(x^{(1)},y^{(1)}\right),\dots,\left(x^{(n)},y^{(n)}\right)\right\}$. (For each $B\geq 0$, if $\hatbb$ is not unique, then we show that the inequalities hold for some choice of $\hatbb$.)
\end{theorem}

In particular, Theorem~\ref{RIP_thm} shows that without placing any assumptions on the covariates outside of $\mathrm{Support}\left(\bstar\right)$, we cannot guarantee a bound on excess error that is better than the optimistic rate obtained by~\citet{SST} from concentration bounds, up to logarithmic factors.

\section{Proofs}

We begin by defining a class of predictors that are optimal with respect to the squared-error loss and the $\ell_1$-norm regularizer:

\begin{definition} Given $y^{(i)}\in\R$ and $x^{(i)}\in\R^p$ for $i=1,\dots,n$, a predictor $\tb\in\R^p$ is Pareto-optimal (with respect to empirical squared-error and $\ell_1$-norm) if it satisfies
$$\sum_i \left(y^{(i)}-x^{(i)}{}^T\b\right)^2\leq \sum_i \left(y^{(i)}-x^{(i)}{}^T\tb\right)^2 \ \Rightarrow \ \|\b\|_1\geq \|\tb\|_1\;,$$
that is, if we cannot improve its empirical squared error without increasing its $\ell_1$-norm, and vice versa.
\end{definition}

The following lemma states a well-known property of $\ell_1$-regularized regression; we include a proof for completeness.
\begin{lemma}\label{lem:Pareto}
For any $y^{(1)},\dots,y^{(n)}\in\R$ and $x^{(1)},\dots,x^{(n)}\in \R^p$, for any $B\geq 0$, the class
$$\mathcal{B}_B\doteq \arg\min_{\|\b\|_1\leq B}\sum_i \left(y^{(i)}-x^{(i)}{}^T\b\right)^2$$
 must contain a predictor $\bhat^B$ that is Pareto-optimal and satisfies $\|\bhat^B\|_0\leq n$. \end{lemma}
\begin{proof}
Let $\mathrm{Err}_B=\inf_{\|\b\|_1\leq B}\sum_i \left(y^{(i)}-x^{(i)}{}^T\b\right)^2$. Since $\left\{\|\b\|_1\leq B\right\}$ is a compact set, this infimum is attained by some $\beta$ with $\|\beta\|_1\leq B$. Now define 
$$B'=\inf\left\{\|\b\|_1:\sum_i \left(y^{(i)}-x^{(i)}{}^T\b\right)^2\leq \mathrm{Err}_B\right\}\leq B\;.$$
Again, by compactness, this infimum is attained by some $\tb$. We then see that $\tb$ is Pareto-optimal by its construction. Finally, by Theorem 3 of \citet{Rosset}, there exists a $\hatbb\in\R^p$ such that $\|\hatbb\|_0\leq n$, $\|\hatbb\|_1\leq \|\tb\|_1$, and $X\hatbb=X\tb$. This is sufficient. 
\end{proof}

Next we state two additional lemmas, proved in the next section.

\begin{lemma}\label{LemmaMax} Fix $n$ and $p$ with $n\geq 30$ and $p\geq 3n$. Let $x^{(i)}\stackrel{i.i.d.}{\sim} N(0,\Sigma)$ for some $\Sigma\in\R^{p\times p}$, and let $\bstar\in\R^p$ be fixed. Then with probability at least $1-2e^{-n\log(p)}$, for all $J\subset [p]$ with $|J|=n$,
\begin{equation}\label{eq:LemmaMax}
\left\|X_J^TX(\tb-\bstar)\right\|_2\leq \|\Sigma\|_{sp}\cdot 
16\sqrt{2}\cdot n\log(p)\cdot\sqrt{(\tb-\bstar)^T\Sigma(\tb-\bstar)}\text{ for all }\tb\in\R^p\text{ with }\tb_{\overline{J}}=0\;,
\end{equation}
where the matrix $X$ has entries $X_{ij}=x^{(i)}_j$, and $X_J$ consists of the columns of $X$ indexed by $j\in J$.
\end{lemma}

\begin{lemma}\label{LemmaMin}
Let $x^{(i)}\stackrel{i.i.d.}{\sim} N(0,\Sigma)$ for some $\Sigma\in\R^{p\times p}$, and let $z\in\R^n$ be fixed, with $\|z\|^2_2\geq 0.5n$. Assume $\nicefrac{\sqrt{n}}{\s}\geq 100$. Then with probability at least $1-e^{-0.015\s^{-1}\sqrt{n}}$, for all $J_1\subset\left[\left\lceil{\nicefrac{\sqrt{n}}{\sigma}}\right\rceil\right]$ with $|J_1|\geq \tfrac{\sqrt{n}}{2\s}$,
\begin{equation} \label{eq:LemmaMin}
 \left\|\mathrm{Proj}_{\mathbf{1}_{J_1}}^{\perp}(X_{J_1}^Tz)\right\|^2_2\geq \frac{\lambda_{\mathrm{min}}^2\left(\Sigma_{J_1}\right)n^{\nicefrac{3}{2}}}{200\s}\;,
\end{equation}
where the matrix $X$ has entries $X_{ij}=x^{(i)}_j$, and $X_{J_1}$ consists of the columns of $X$ indexed by $j\in J_1$.
\end{lemma}

We now prove the theorems.

\subsection{Proof of Theorem~\ref{sparsity_thm}}

\begin{reptheorem}{sparsity_thm} Fix any $n\geq 30$, $p\geq 3n$, and $\sigma\geq 0$. Then there exists a $\beta^*\in\R^p$ with $\frac{1}{2}\leq\|\bstar\|_2\leq\|\beta^*\|_1\leq 1$, such that for any sample, for all $B\geq 0$,
\begin{equation}\label{sparsity_FastRate}
\EE{\left(y-x^T\hatbb\right)^2}\geq \sigma^2+ \frac{1}{32n\log^2(3n)}\;.
\end{equation}
Additionally, if $100\leq \nicefrac{\sqrt{n}}{\s}\leq p$, then with probability at least $\frac{1}{2}$ over the sample, for all $B\geq 0$,
\begin{equation}\label{sparsity_SlowRate}
\EE{\left(y-x^T\hatbb\right)^2}\geq \sigma^2+ \frac{\s}{102400\sqrt{n}\log^2\left(\max\left\{3n,\left\lceil{\nicefrac{\sqrt{n}}{\sigma}}\right\rceil\right\}\right)}\;.
\end{equation}
Here $\hatbb=\arg\min_{\|\beta\|_1\leq B}\sum_i \left(y^{(i)}-x^{(i)}{}^T\beta\right)^2$, where $\left(x^{(i)},y^{(i)}\right)$ are i.i.d.~samples from the multivariate Gaussian distribution defined by drawing $x^{(i)}\sim N(0,\mathbf{I}_p)$ and $y^{(i)}\sim N\left(x^{(i)}{}^T\bstar,\s^2\right)$.  The expectations are taken over a new sample $(x,y)$ drawn from the same distribution, independently of the training set $\left\{\left(x^{(1)},y^{(1)}\right),\dots,\left(x^{(n)},y^{(n)}\right)\right\}$. (For each $B\geq 0$, if $\hatbb$ is not unique, then we show that the inequalities hold for some choice of $\hatbb$.)

\end{reptheorem}

\begin{proof}

 Let  $\bstar$ be
$$\bstar_j=\frac{1}{j\cdot 4\log p}, \ j=1,\dots,p-1;  \ \bstar_p=\frac{1}{2}\;.$$
Note that $\|\bstar\|_1\leq 1$ and $\|\bstar\|^2_2\geq\frac{1}{4}$, and so the resulting distribution satisfies the desired assumptions.t

By Lemma~\ref{lem:Pareto}, for any $B\geq 0$, the set $\arg\min_{\|\b\|_1\leq B}\sum_i \left(y^{(i)}-x^{(i)}{}^T\b\right)^2$ must include a Pareto-optimal vector $\bhat^B$ with $\|\bhat^B\|_0\leq n$. Therefore, it is sufficient to show that bounds~(\ref{sparsity_FastRate}) and~(\ref{sparsity_SlowRate}) hold for all Pareto-optimal vectors $\bhat$ with $\|\bhat\|_0\leq n$. We now prove these two bounds separately.

\paragraph{Proof of~(\ref{sparsity_FastRate}).} For any $\bhat$ with $\|\bhat\|_0\leq n$, we have
\begin{align*}
&\left\|\bhat-\bstar\right\|^2_2\geq \sum_{j=1}^{p}\left(\bhat - \frac{1}{j\cdot 4\log p}\right)^2
\geq \sum_{j:\bhat_j=0}\left(\frac{1}{j\cdot 4\log p}\right)^2\geq \sum_{j={n+1}}^p \left(\frac{1}{j\cdot 4\log p}\right)^2\\
&\geq \frac{1}{16\log^2(p)} \int_{x=n+1}^p \frac{1}{x^2}\;dx =\frac{1}{16\log^2(p)}\left(\frac{1}{n+1}-\frac{1}{p}\right)\geq \frac{1}{32n\log^2(p)}\;.
\end{align*}

This proves the claim when $p=3n$. However, the claim is immediately true for any larger value of $p$, since we may add in an arbitrary number of zero covariates (and assign zero coefficients to these covariates), without affecting the results.

\paragraph{Proof of~(\ref{sparsity_SlowRate}).} By Lemma 1 of \citet{LaurentMassart}, with probability at least $1-e^{-0.0625n}\geq 0.75$, $\|z\|^2_2\sim \chi^2_n\geq 0.5n$. For the remainder of the proof, we treat $z\in\R^n$ as a fixed vector, and assume $\|z\|^2_2\geq 0.5n$.

Assume that~(\ref{eq:LemmaMax}) holds for all $J\subset[p]$ with $|J|=n$, and~(\ref{eq:LemmaMin}) holds for all $J_1\subset\left[\left\lceil\nicefrac{\sqrt{n}}{\s}\right\rceil\right]$ with $|J_1|\geq \nicefrac{\sqrt{n}}{2\s}$. (By Lemmas~\ref{LemmaMax} and~\ref{LemmaMin}, this is true with probability at least $1-2e^{-n\log(p)}-e^{-0.015\sigma^{-1}\sqrt{n}}\geq 0.75$.) Now choose any Pareto-optimal $\bhat$ with $\|\bhat\|_0\leq n$.

Suppose that $ \|\bhat-\bstar\|^2_2<\frac{\sigma}{102400\sqrt{n}\log^2 (p)}\;.$
First, we show that 
$$\left|\left\{j\in\left[\left\lceil{\nicefrac{\sqrt{n}}{\sigma}}\right\rceil\right]:\bhat_j>0\right\}\right|\geq  \tfrac{\sqrt{n}}{2\s}\;.$$
Suppose not. Then
\begin{align*}
\|\bhat-\bstar\|^2_2&\geq
\sum_{j\in\left[\left\lceil{\nicefrac{\sqrt{n}}{\sigma}}\right\rceil\right]:\bhat_j\leq 0}(\bstar_j)^2
=\tfrac{1}{16\log^2(p)}\sum_{j\in\left[\left\lceil{\nicefrac{\sqrt{n}}{\sigma}}\right\rceil\right]:\bhat_j\leq 0}\tfrac{1}{j^2}
\geq \tfrac{1}{16\log^2(p)}\sum_{j=\left\lceil{\nicefrac{\sqrt{n}}{2\s}}\right\rceil}^{\left\lceil{\nicefrac{\sqrt{n}}{\sigma}}\right\rceil}\tfrac{1}{j^2}\\
&\geq \tfrac{1}{16\log^2(p)}\int_{x=\left\lceil{\nicefrac{\sqrt{n}}{2\s}}\right\rceil}^{2\left\lceil{\nicefrac{\sqrt{n}}{2\s}}\right\rceil}\tfrac{1}{x^2}dx
= \tfrac{1}{16\log^2(p)}\left(\tfrac{1}{\left\lceil{\nicefrac{\sqrt{n}}{2\s}}\right\rceil}-\tfrac{1}{2\left\lceil{\nicefrac{\sqrt{n}}{2\s}}\right\rceil}\right)
= \tfrac{1}{16\log^2(p)\cdot 2\left\lceil{\nicefrac{\sqrt{n}}{2\s}}\right\rceil}\geq\tfrac{\sigma}{32\sqrt{n}\log^2(p)}\;.\end{align*}
This is a contradiction.

Now define $J_1=\left\{j\in\left[\left\lceil{\nicefrac{\sqrt{n}}{\sigma}}\right\rceil\right]:\bhat_j>0\right\}$, and fix any $J\supset \mathrm{Support}(\bhat)$ with $|J|=n$. Since $\bhat$ is Pareto-optimal with positive entries $\bhat_j$ for all $j\in J_1$, we have 
$$\frac{\partial}{\partial \left(\b_{J_1}\right)}\|\bhat\|_1=\mathbf{1}_{J_1}\;.$$
Therefore, by the theory of Lagrange multipliers, we must have
$X_{J_1}^Ty-X_{J_1}^TX\bhat = C\cdot\mathbf{1}_{J_1}$, for some $C\in\R$.
We then have
\begin{equation}\label{eq:Lagrange1}X_{J_1}^TX(\bhat-\bstar)=\sigma\cdot X_{J_1}^Tz- C\cdot\mathbf{1}_{J_1}\;.\end{equation}

By~(\ref{eq:LemmaMax}), the norm of the left-hand side of~(\ref{eq:Lagrange1}) can be bounded from above as
$$\left\|X_{J_1}^TX(\bhat-\bstar)\right\|_2\leq \left\|X_{J_1}^TX(\bhat-\bstar)\right\|_2\leq \|\Sigma\|_{sp}\cdot 16\sqrt{2}\cdot n\log(p)\cdot \sqrt{(\bhat-\bstar)^T\Sigma(\bhat-\bstar)}\;.$$

By~(\ref{eq:LemmaMin}), the norm of the right-hand side of~(\ref{eq:Lagrange1}) can be bounded from below as
$$\left\|\sigma\cdot X_{J_1}^Tz- C\cdot\mathbf{1}_{J_1}\right\|_2\geq \sigma\cdot \left\|\text{Proj}_{\mathbf{1}_{J_1}^{\perp}}X_{J_1}^Tz\right\|_2\geq \s\sqrt{ \frac{\lambda_{\mathrm{min}}^2\left(\Sigma_{J_1}\right)n^{\nicefrac{3}{2}}}{200\s}}\;.$$

Therefore, returning to~(\ref{eq:Lagrange1}), we have
\begin{align*}
&\|\Sigma\|_{sp}\cdot 16\sqrt{2}\cdot n\log(p)\cdot \sqrt{(\bhat-\bstar)^T\Sigma(\bhat-\bstar)}\geq\left\|X_{J_1}^TX(\bhat-\bstar)\right\|_2\\
&\hspace{2cm}=\left\|\sigma\cdot X_{J_1}^Tz- C\cdot\mathbf{1}_{J_1}\right\|_2\geq \s\sqrt{ \frac{\lambda_{\mathrm{min}}^2\left(\Sigma_{J_1}\right)n^{\nicefrac{3}{2}}}{200\s}}\;.\end{align*}
Therefore,
$$(\bhat-\bstar)^T\Sigma(\bhat-\bstar)\geq  \frac{\s\cdot \lambda_{\mathrm{min}}^2\left(\Sigma_{J_1}\right)}{102400\|\Sigma\|^2_{sp}\cdot \sqrt{n}\log^2(p)}= \frac{\s}{102400 \sqrt{n}\log^2(p)}\;.$$
This proves the claim when $p=\max\left\{3n,\left\lceil{\nicefrac{\sqrt{n}}{\s}}\right\rceil\right\}$. As in the proof of~(\ref{sparsity_FastRate}), this is sufficient to prove the claim for any larger value of $p$.

\end{proof}

\subsection{Proof of Theorem~\ref{RIP_thm}}

\begin{reptheorem}{RIP_thm} Fix any $n\geq 30$, $p\geq 3n$, and $\sigma\geq 0$. Then there exists a $2$-sparse $\beta^*\in\R^p$ with $\frac{1}{2}\leq\|\bstar\|_2\leq\|\beta^*\|_1\leq 1$, and a positive semi-definite $\Sigma\in\R^{p\times p}$ with $\|\Sigma\|_{sp}\leq 2$ and $\Sigma_{\mathrm{Support}\left(\bstar\right)}=\mathbf{I}_{\mathrm{Support}\left(\bstar\right)}$, such that for any sample, for all $B\geq 0$,
\begin{equation}\label{RIP_FastRate}
\EE{\left(y-x^T\hatbb\right)^2}\geq \sigma^2+ \frac{1}{288n\log^2(3n)}\;.
\end{equation}
Additionally, if $100\leq \nicefrac{\sqrt{n}}{\s}\leq p-3$, then with probability at least $\frac{1}{2}$ over the sample, for all $B\geq 0$,
\begin{equation}\label{RIP_SlowRate}
\EE{\left(y-x^T\hatbb\right)^2}\geq \sigma^2+ \frac{\s}{409600\sqrt{n}\log^2\left(\max\left\{3n,\left\lceil{\nicefrac{\sqrt{n}}{\sigma}}\right\rceil\right\}\right)}\;.
\end{equation}
Here $\hatbb=\arg\min_{\|\beta\|_1\leq B}\sum_i \left(y^{(i)}-x^{(i)}{}^T\beta\right)^2$, where $\left(x^{(i)},y^{(i)}\right)$ are i.i.d.~samples from the multivariate Gaussian distribution defined by drawing $x^{(i)}\sim N(0,\Sigma)$ and $y^{(i)}\sim N\left(x^{(i)}{}^T\bstar,\s^2\right)$. The expectations are taken over a new sample $(x,y)$ drawn from the same distribution, independently of the training set $\left\{\left(x^{(1)},y^{(1)}\right),\dots,\left(x^{(n)},y^{(n)}\right)\right\}$. (For each $B\geq 0$, if $\hatbb$ is not unique, then we show that the inequalities hold for some choice of $\hatbb$.)\end{reptheorem}

\begin{proof}

Let $w_1,w_2,u_1,\dots,u_{p-3}\stackrel{iid}{\sim}N(0,1)$. Define
$$\tau=\frac{1}{4\log p}\cdot \left(\frac{1}{1},\frac{1}{2},\dots,\frac{1}{p-3}\right)\in\R^{p-3}\;.$$\
Since $p\geq 90$, $\|\tau\|_1\leq \frac{1}{3}$ and $\|\tau\|^2_2<\frac{1}{9\log^2(p)}\leq 0.01$. Now we define an additional covariate as a linear combination of the others:
$$v=\frac{1}{\sqrt{2}}\left(w_1+w_2\right)\cdot \sqrt{1-\|\tau\|^2_2} - u^T\tau\;.$$

Now define $x=(u_1,\dots,u_{p-3},v,w_1,w_2)$. Let $\Sigma=Cov(x)$, and note that $\smax=\|\S\|_2\leq 2$.

Define
 $$\bstar_{sparse}=\left(\mathbf{0}_{p-3},0,\frac{1}{2},\frac{1}{2}\right), \  \bstar_{dense}=\left(\frac{1}{\sqrt{2(1-\|\tau\|^2_2)}}\cdot \tau,\frac{1}{\sqrt{2(1-\|\tau\|^2_2)}},0,0\right)\;.$$
 and
 $$y^{(i)}=\frac{1}{2}\left(w_1+w_2\right)=x^{(i)}{}^T\bstar_{sparse}=x^{(i)}{}^T\bstar_{dense}\;.$$

  Note that $\bstar_{sparse}$ and $\bstar_{dense}$ are both optimal predictors. Since $\|\bstar_{sparse}\|_1=1$, $\bstar_{sparse}{}^T\Sigma\bstar_{sparse}=\|\bstar_{sparse}\|^2_2= \frac{1}{2}$, and $\bstar_{sparse}$ is $2$-sparse, this distribution satisfies the desired assumptions. However, $$  \|\bstar_{dense}\|_1=\frac{1}{\sqrt{2(1-\|\tau\|^2_2)}}(1+\|\tau\|_1)\approx \frac{4}{3\sqrt{2}}<1\;,$$
and so in a sense $\bstar_{dense}$ will be preferred to $\bstar_{sparse}$ in $\ell_1$-regularized regression, thus leading to the same arguments as in the proof of Theorem~\ref{sparsity_thm}.

By Lemma~\ref{lem:Pareto}, for any $B\geq 0$, the set $\arg\min_{\|\b\|_1\leq B}\sum_i \left(y^{(i)}-x^{(i)}{}^T\b\right)^2$ must include a Pareto-optimal vector $\bhat^B$ with $\|\bhat^B\|_0\leq n$. Therefore, it is sufficient to show that bounds~(\ref{RIP_FastRate}) and~(\ref{RIP_SlowRate}) hold for all Pareto-optimal vectors $\bhat$ with $\|\bhat\|_0\leq n$. For each such $\bhat$, we use the notation
$$\bhat=(\bhat_u,\bhat_{w_1},\bhat_{w_2},\bhat_v)\in\R^{p-3}\times\R\times\R\times\R\;.$$

Observe that, by definition of the covariates, 
\begin{align}
{\label{eq:Term123}}&(\bhat-\bstar_{sparse})^T{\Sigma}(\bhat-\bstar_{sparse})\\
\notag&=\underbrace{\left\|\bhat_u-\tau\bhat_v\right\|^2_2}_{\text{(Term 1)}}+\underbrace{\left(\bhat_{w_1}+\frac{1}{\sqrt{2}}\sqrt{1-\|\tau\|^2_2}\bhat_v-\frac{1}{\sqrt{2}}\right)^2}_{\text{(Term 2)}}+\underbrace{\left(\bhat_{w_2}+\frac{1}{\sqrt{2}}\sqrt{1-\|\tau\|^2_2}\bhat_v-\frac{1}{\sqrt{2}}\right)^2}_{\text{(Term 3)}}\;.
\end{align}

The remainder of the proof is organized as follows. First, we prove bounds~(\ref{RIP_FastRate}) and~(\ref{RIP_SlowRate}) for any Pareto-optimal $\bhat$ with $\bhat_v\leq \frac{1}{3}$. Next, we prove the bound~(\ref{RIP_FastRate}) for any Pareto-optimal $\bhat$ with $\|\bhat\|_0\leq n$ and $\bhat_v> \frac{1}{3}$. Finally, we prove the bound~(\ref{RIP_SlowRate}) for any Pareto-optimal $\bhat$ with $\|\bhat\|_0\leq n$ and $\bhat_v> \frac{1}{3}$.  

\paragraph{Proof of~(\ref{RIP_FastRate}) and~(\ref{RIP_SlowRate}) when $\bhat_v\leq \frac{1}{3}$.}

Consider any Pareto-optimal $\bhat$ with $\bhat_v\leq \frac{1}{3}$. First, suppose that $\bhat_{w_1},\bhat_{w_2}\geq \frac{1}{2\sqrt{2}}$. Let
$$\tilde{\beta}=\left(\bhat_u+\tfrac{1}{2\sqrt{1-\|\tau\|^2_2}}\cdot \tau,\bhat_{w_1}-\tfrac{1}{2\sqrt{2}},\bhat_{w_2}-\tfrac{1}{2\sqrt{2}},\bhat_v+\tfrac{1}{2\sqrt{1-\|\tau\|^2_2}}\right)\;.$$
By the definition of the covariates, $x^{(i)}{}^T\bhat=x^{(i)}{}^T\tilde{\beta}$ for all $i$. We will now show that $\|\tb\|_1<\|\bhat\|_1$. We have
\begin{align*}
\|\tilde{\beta}\|_1&=\|\tb_u\|_1+|\tb_{w_1}|+|\tb_{w_2}|+|\tb_v|\\
&=\left\|\bhat_u+\tfrac{1}{2\sqrt{1-\|\tau\|^2_2}}\cdot \tau\right\|_1+\left|\bhat_{w_1}-\tfrac{1}{2\sqrt{2}}\right|+\left|\bhat_{w_2}-\tfrac{1}{2\sqrt{2}}\right|+\left|\bhat_v+\tfrac{1}{2\sqrt{1-\|\tau\|^2_2}}\right|\\
&=\left\|\bhat_u+\tfrac{1}{2\sqrt{1-\|\tau\|^2_2}}\cdot \tau\right\|_1+\left|\bhat_{w_1}\right|+\left|\bhat_{w_2}\right|-\tfrac{1}{\sqrt{2}}+\left|\bhat_v+\tfrac{1}{2\sqrt{1-\|\tau\|^2_2}}\right|\\
&\leq\left\|\bhat_u\right\|_1+\tfrac{1}{2\sqrt{1-\|\tau\|^2_2}}\cdot \|\tau\|_1+\left|\bhat_{w_1}\right|+\left|\bhat_{w_2}\right|-\tfrac{1}{\sqrt{2}}+\left|\bhat_v\right|+\tfrac{1}{2\sqrt{1-\|\tau\|^2_2}}\\
&=\|\bhat\|_1-\tfrac{1}{\sqrt{2}}+\tfrac{1}{2\sqrt{1-\|\tau\|^2_2}}+\tfrac{\|\tau\|_1}{2\sqrt{1-\|\tau\|^2_2}}
\leq \|\bhat\|_1-\tfrac{1}{\sqrt{2}}+\tfrac{1}{2\sqrt{1-0.01^2}}+\tfrac{0.3}{2\sqrt{1-0.01^2}}
\leq \|\bhat\|_1-0.05\;.
\end{align*}
Therefore, this case leads to a contradiction, since we have constructed a coefficient vector $\tb$ with zero error on the training set, and lower $\ell_1$-norm than $\bhat$. Therefore, we must have either $\bhat_{w_1}<\frac{1}{2\sqrt{2}}$ or  $\bhat_{w_2}<\frac{1}{2\sqrt{2}}$. Without loss of generality, we assume  $\bhat_{w_1}<\frac{1}{2\sqrt{2}}$.

 Then
$$\bhat_{w_1}+\frac{1}{\sqrt{2}}\sqrt{1-\|\tau\|^2_2}\bhat_v-\frac{1}{\sqrt{2}}<\frac{1}{2\sqrt{2}}+\frac{1}{\sqrt{2}}\cdot \frac{1}{3}-\frac{1}{\sqrt{2}}\leq -\frac{1}{6\sqrt{2}}\;,$$
and so by (Term 2) in~(\ref{eq:Term123}) above,
$$(\bhat-\bstar_{sparse})^T{\Sigma}(\bhat-\bstar_{sparse})\geq\left(\bhat_{w_1}+\frac{1}{\sqrt{2}}\sqrt{1-\|\tau\|^2_2}\bhat_v-\frac{1}{\sqrt{2}}\right)^2\geq \frac{1}{72}\;.$$

This is sufficient to show that both~(\ref{RIP_FastRate}) and~(\ref{RIP_SlowRate}) are satisfied.

\paragraph{Proof of~(\ref{RIP_FastRate}) when $\bhat_v>\frac{1}{3}$.}

Consider any Pareto-optimal $\bhat$ with $\|\bhat\|_0\leq n$ and $\bhat_v> \frac{1}{3}$. We then have
\begin{align*}
&\left\|\bhat_u-\tau\bhat_v\right\|^2_2
=\sum_{j=1}^{p-3}\left(\bhat_{u_j}-\frac{1}{4\log(p)}\cdot\frac{1}{j}\cdot \bhat_v\right)^2
\geq \sum_{j\in\{1,\dots,p-3\}:\bhat_{u_j}=0}\left(\frac{1}{4\log(p)}\cdot\frac{1}{j}\cdot \bhat_v\right)^2\\
&\geq\frac{\bhat_v^2}{16\log^2(p)}\cdot  \sum_{j=n}^{p-3}\frac{1}{j^2}
\geq\frac{1}{144\log^2(p)}\cdot \int_{x=n}^{p-3}\frac{1}{x^2}\;dx
=\frac{1}{144\log^2(p)}\cdot\left(\frac{1}{n}-\frac{1}{p-3}\right)\geq \frac{1}{288n\log^2(p)}\;.
\end{align*}

But, considering (Term 1) in~(\ref{eq:Term123}) above, this proves that
$$(\bhat-\bstar_{sparse})^T{\Sigma}(\bhat-\bstar_{sparse})\geq\frac{1}{288n\log^2(p)}\;.$$

This proves the claim when $p=3n$. As in the proof of Theorem~\ref{sparsity_thm}, this is sufficient to prove the claim for any larger value of $p$.

\paragraph{Proof of~(\ref{RIP_SlowRate}) when $\bhat_v>\frac{1}{3}$.}
By Lemma 1 of \citet{LaurentMassart}, with probability at least $1-e^{-0.0625n}\geq 0.75$, $\|z\|^2_2\sim \chi^2_n\geq 0.5n$. For the remainder of the proof, we treat $z\in\R^n$ as a fixed vector, and assume $\|z\|^2_2\geq 0.5n$.

Assume that~(\ref{eq:LemmaMax}) holds for all $J\subset[p]$ with $|J|=n$, and~(\ref{eq:LemmaMin}) holds for all $J_1\subset\left[\left\lceil\nicefrac{\sqrt{n}}{\s}\right\rceil\right]$ with $|J_1|\geq \nicefrac{\sqrt{n}}{2\s}$. (By Lemmas~\ref{LemmaMax} and~\ref{LemmaMin}, this is true with probability at least $1-2e^{-n\log(p)}-e^{-0.015\sigma^{-1}\sqrt{n}}\geq 0.75$.) 
Consider any Pareto-optimal $\bhat$ with $\|\bhat\|_0\leq n$ and $\bhat_v> \frac{1}{3}$. First, suppose that 
$$\left|\left\{j\in\left[\left\lceil{\nicefrac{\sqrt{n}}{\sigma}}\right\rceil\right]:\bhat_j>0\right\}\right|<  \tfrac{\sqrt{n}}{2\s}\;.$$
 Then \begin{align*}
&\|\bhat_u-\tau\bhat_v\|^2_2
=\sum_{j=1}^{p-3}\left(\bhat_{u_j}-\frac{1}{4\log(p)}\cdot\frac{1}{j}\cdot \bhat_v\right)^2
\geq
\sum_{j\in\left[\left\lceil{\nicefrac{\sqrt{n}}{\sigma}}\right\rceil\right]:\bhat_{u_j}\leq 0}\left(\frac{1}{4\log(p)}\cdot\frac{1}{j}\cdot \bhat_v\right)^2\\
&=\tfrac{\bhat_v{}^2}{16\log^2(p)}\sum_{j\in\left[\left\lceil{\nicefrac{\sqrt{n}}{\sigma}}\right\rceil\right]:\bhat_{u_j}\leq 0}\tfrac{1}{j^2}
\geq \tfrac{1}{144\log^2(p)}\sum_{j=\left\lceil{\nicefrac{\sqrt{n}}{2\s}}\right\rceil}^{\left\lceil{\nicefrac{\sqrt{n}}{\sigma}}\right\rceil}\tfrac{1}{j^2}\\
&\geq \tfrac{1}{144\log^2(p)}\int_{x=\left\lceil{\nicefrac{\sqrt{n}}{2\s}}\right\rceil}^{2\left\lceil{\nicefrac{\sqrt{n}}{2\s}}\right\rceil}\tfrac{1}{x^2}dx
= \tfrac{1}{144\log^2(p)}\left(\tfrac{1}{\left\lceil{\nicefrac{\sqrt{n}}{2\s}}\right\rceil}-\tfrac{1}{2\left\lceil{\nicefrac{\sqrt{n}}{2\s}}\right\rceil}\right)
= \tfrac{1}{144\log^2(p)\cdot 2\left\lceil{\nicefrac{\sqrt{n}}{2\s}}\right\rceil}\geq\tfrac{\s}{288\sqrt{n}\log^2(p)}\;.\end{align*}

Considering (Term 1) in~(\ref{eq:Term123}), this proves that
$$(\bhat-\bstar_{sparse})^T{\Sigma}(\bhat-\bstar_{sparse})\geq \left\|\bhat_u-\tau\bhat_v\right\|^2_2\geq \tfrac{\s}{288\sqrt{n}\log^2(p)}\;.$$

Next, suppose instead that
$$\left|\left\{j\in\left[\left\lceil{\nicefrac{\sqrt{n}}{\sigma}}\right\rceil\right]:\bhat_j>0\right\}\right|\geq  \tfrac{\sqrt{n}}{2\s}\;.$$

Define $J_1=\left\{j\in\left[\left\lceil{\nicefrac{\sqrt{n}}{\sigma}}\right\rceil\right]:\bhat_j>0\right\}$, and fix any $J\supset \mathrm{Support}(\bhat)$ with $|J|=n$. Since $\bhat$ is Pareto-optimal with positive entries $\bhat_j$ for all $j\in J_1$, we have 
$$\frac{\partial}{\partial \left(\b_{J_1}\right)}\|\bhat\|_1=\mathbf{1}_{J_1}\;.$$
Therefore, by the theory of Lagrange multipliers, we must have
$X_{J_1}^Ty-X_{J_1}^TX\bhat = C\cdot\mathbf{1}_{J_1}$, for some $C\in\R$.
We then have
\begin{equation}\label{eq:Lagrange2}X_{J_1}^TX(\bhat-\bstar)=\sigma\cdot X_{J_1}^Tz- C\cdot\mathbf{1}_{J_1}\;.\end{equation}

By~(\ref{eq:LemmaMax}), the norm of the left-hand side of~(\ref{eq:Lagrange2}) can be bounded from above as
$$\left\|X_{J_1}^TX(\bhat-\bstar)\right\|_2\leq \left\|X_{J_1}^TX(\bhat-\bstar)\right\|_2\leq \|\Sigma\|_{sp}\cdot 16\sqrt{2}\cdot n\log(p)\cdot \sqrt{(\bhat-\bstar)^T\Sigma(\bhat-\bstar)}\;.$$

By~(\ref{eq:LemmaMin}), the norm of the right-hand side of~(\ref{eq:Lagrange2}) can be bounded from below as
$$\left\|\sigma\cdot X_{J_1}^Tz- C\cdot\mathbf{1}_{J_1}\right\|_2\geq \sigma\cdot \left\|\text{Proj}_{\mathbf{1}_{J_1}^{\perp}}X_{J_1}^Tz\right\|_2\geq \s\sqrt{ \frac{\lambda_{\mathrm{min}}^2\left(\Sigma_{J_1}\right)n^{\nicefrac{3}{2}}}{200\s}}\;.$$

Therefore, returning to~(\ref{eq:Lagrange2}), we have
\begin{align*}
&\|\Sigma\|_{sp}\cdot 16\sqrt{2}\cdot n\log(p)\cdot \sqrt{(\bhat-\bstar)^T\Sigma(\bhat-\bstar)}\geq\left\|X_{J_1}^TX(\bhat-\bstar)\right\|_2\\
&\hspace{2cm}=\left\|\sigma\cdot X_{J_1}^Tz- C\cdot\mathbf{1}_{J_1}\right\|_2\geq \s\sqrt{ \frac{\lambda_{\mathrm{min}}^2\left(\Sigma_{J_1}\right)n^{\nicefrac{3}{2}}}{200\s}}\;.\end{align*}
Therefore,
$$(\bhat-\bstar)^T\Sigma(\bhat-\bstar)\geq  \frac{\s\cdot \lambda_{\mathrm{min}}^2\left(\Sigma_{J_1}\right)}{102400\|\Sigma\|^2_{sp}\cdot \sqrt{n}\log^2(p)}= \frac{\s}{409600 \sqrt{n}\log^2(p)}\;.$$

This proves the claim when $p=\max\left\{3n,\left\lceil{\nicefrac{\sqrt{n}}{\s}}\right\rceil\right\}$. As in the proof of Theorem~\ref{sparsity_thm}, this is sufficient to prove the claim for any larger value of $p$.

\end{proof}

\section{Proofs for Lemmas}

\subsection{Proof of Lemma~\ref{LemmaMax}}
Fix any $J\subset [p]$ with $|J|=n$. We will show that, with probability at least $1-2e^{-2n\log(p)}$, 
$$\left\|X_J^TX(\tb-\bstar)\right\|_2\leq \|\Sigma\|_{sp}\cdot 
16\sqrt{2}\cdot n\log(p)\cdot\sqrt{(\tb-\bstar)^T\Sigma(\tb-\bstar)}\text{ for all }\tb\in\R^p\text{ with }\tb_{\overline{J}}=0\;.
$$
Since there are  ${p\choose n}\leq p^n$ choices for the set $J$, this will be sufficient to prove the lemma.

Reorder the covariates to write $\Sigma=\left(\begin{array}{cc}\Sigma_{JJ} &\Sigma_{J\overline{J}} \\ \Sigma_{\overline{J}J} & \Sigma_{\overline{J}\overline{J}}\\\end{array}\right)$. Choose a Cholesky decomposition  
$$\Sigma=\left(\begin{array}{cc} U&V\\\mathbf{0}&W\\\end{array}\right)^T\left(\begin{array}{cc} U&V\\\mathbf{0}&W\\\end{array}\right)\;.$$
Let $a^{(i)}\sim N(0,\mathbf{I}_p)$. Then  
$\left(\begin{array}{cc} U&V\\\mathbf{0}&W\\\end{array}\right)^Ta^{(i)}\sim N(0,\Sigma)$, and so
$$\left(\begin{array}{cc}X_J&X_{\overline{J}}\\\end{array}\right)\stackrel{\mathcal{D}}{=}\left(\begin{array}{cc} A_JU & A_JV+A_{\overline{J}}W\\\end{array}\right)\;,$$
where the matrix $A$ has entries $A_{ij}=a^{(i)}_j$, and $A_J$ consists of the columns of $A$ indexed by $j\in J$. 
We then have
\begin{align*}
&X_J^TX(\tb-\bstar)
\stackrel{\mathcal{D}}{=}U^TA_J^T\left(A_JU(\tb-\bstar)_J+(A_JV+A_{\overline{J}}W)(\tb-\bstar)_{\overline{J}}\right)\\
&=U^TA_J^T\left(A_JU(\tb-\bstar)_J-(A_JV+A_{\overline{J}}W)\bstar_{\overline{J}}\right)
=U^TA_J^T\left(A_J\left(U(\tb-\bstar)_J-V\bstar_{\overline{J}}\right)-A_{\overline{J}}W\bstar_{\overline{J}}\right)\\
\end{align*}

Below, we will show that
\begin{align}
\label{eq:LemmaMax_1}&\|A_J\|_{sp}\leq \sqrt{16n\log(p)}\text{ with probability at least }1-e^{-2n\log(p)}\;,\\
\label{eq:LemmaMax_2}&\text{and }\left\|A_{\overline{J}}W\bstar_{\overline{J}}\right\|_2\leq \sqrt{16n\log(p)}\cdot \left\|W\bstar_{\overline{J}}\right\|_2\text{ with probability at least }1-e^{-2n\log(p)}\;.
\end{align}

Assuming that these bounds hold. Then for any $\tb\in\R^p$ with $\tb_{\overline{J}}=0$, we have
\begin{align*}
&\left\|U^TA_J^T\left(A_J\left(U(\tb-\bstar)_J-V\bstar_{\overline{J}}\right)-A_{\overline{J}}W\bstar_{\overline{J}}\right)\right\|_2\\
&\leq \|U\|_{sp}\cdot \|A_J\|_{sp}\cdot \left(\|A_J\|_{sp}\cdot \left\|U(\tb-\bstar)_J-V\bstar_{\overline{J}}\right\|_2+ \left\|A_{\overline{J}}W\bstar_{\overline{J}}\right\|_2\right)\\
&\leq \|\Sigma\|_{sp}\cdot \sqrt{16n\log(p)}\cdot \left( \sqrt{16n\log(p)}\cdot \left\|U(\tb-\bstar)_J-V\bstar_{\overline{J}}\right\|_2+\sqrt{16n\log(p)}\cdot  \left\|W\bstar_{\overline{J}}\right\|_2\right)\\
&= \|\Sigma\|_{sp}\cdot{16n\log(p)}\cdot \left(\left\|U(\tb-\bstar)_J-V\bstar_{\overline{J}}\right\|_2+\left\|W\bstar_{\overline{J}}\right\|_2\right)\\
&\leq \|\Sigma\|_{sp}\cdot{16n\log(p)}\cdot \sqrt{2}\cdot \left(\left\|U(\tb-\bstar)_J-V\bstar_{\overline{J}}\right\|^2_2+\left\|W\bstar_{\overline{J}}\right\|^2_2\right)^{\nicefrac{1}{2}}\\
&=\|\Sigma\|_{sp}\cdot{16\sqrt{2}\cdot n\log(p)}\cdot\sqrt{(\tb-\bstar)^T\Sigma(\tb-\bstar)}\;.
\end{align*}

We conclude by proving~(\ref{eq:LemmaMax_1}) and~(\ref{eq:LemmaMax_2}). We first prove~(\ref{eq:LemmaMax_1}) using a  construction from \citet{KMO}.
First, define $\U=\left\{u\in\left(\tfrac{1}{8\sqrt{n}}\Z\right)^n:\|u\|_2\leq 1\right\}$. By Remark 5.1 in \citet{KMO},
$$\left\|A_J\right\|_{sp}\leq \sqrt{2}\sup_{u,v\in \U}\left|u^TA_Jv\right|\;.$$
For any $u,v\in \U$, 
$$u^TA_Jv=\sum_{i=1}^n\sum_{j\in J}u_iv_jA_{ij}\sim N(0,\|u\|^2_2\|v\|^2_2)\;,$$
therefore,
$$Pr\left(\left|u^TA_Jv\right|\geq \sqrt{8 n\log(p)}\right)\leq Pr\left(|N(0,1)|\geq \sqrt{8 n\log(p)}\right)\leq e^{-4n\log(p)}\;.$$

Furthermore, $|\U|\leq \left(2\left\lceil{8\sqrt{n}}\right\rceil+1\right)^n\leq p^n$. 
So,
\begin{align*}
&Pr\left(\left\|A_J\right\|_{sp}\leq  \sqrt{16n\log(p)}\right)\leq Pr\left(\left|u^TA_Jv\right|\leq \sqrt{8\cdot n\log(p)}\text{ for all }u,v\in\U\right)\\
&\geq 1-p^{2n}e^{-4n\log(p)}\geq 1-e^{-2n\log(p)}\;.
\end{align*}

Next we prove~(\ref{eq:LemmaMax_2}). We have
\begin{align*}
&\left\|A_{\overline{J}}W\bstar_{\overline{J}}\right\|^2_2
=\sum_i \left(\sum_{j\in\overline{J}}a^{(i)}_j \left(W\bstar_{\overline{J}}\right)_j\right)^2
\stackrel{\mathcal{D}}{=}\left\|W\bstar_{\overline{J}}\right\|^2_2\cdot \chi^2_n\;.
\end{align*}
By Lemma 1 of \citet{LaurentMassart},
$Pr\left(\chi^2_n\geq 16n\log(p)\right)\leq  e^{-2n\log(p)}$. This is sufficient.

\subsection{Proof of Lemma~\ref{LemmaMin}}

Choose any $J_2\subset J_1$ with $|J_2|=\left\lceil{\nicefrac{\sqrt{n}}{2\s}}\right\rceil$. Observe that
$\left\|\mathrm{Proj}_{\mathbf{1}_{J_1}}^{\perp}(X_{J_1}^Tz)\right\|^2_2\geq \left\|\mathrm{Proj}_{\mathbf{1}_{J_2}}^{\perp}(X_{J_2}^Tz)\right\|^2_2$,
and so it is sufficient to only consider the sets $J_2$ of size $\left\lceil{\nicefrac{\sqrt{n}}{2\s}}\right\rceil$.

Fix any $J_2\subset \left[\left\lceil{\nicefrac{\sqrt{n}}{\sigma}}\right\rceil\right]$ with $|J_2|=\left\lceil{\nicefrac{\sqrt{n}}{2\sigma}}\right\rceil$. Let $P\in\R^{\left\lceil{\nicefrac{\sqrt{n}}{2\sigma}}\right\rceil\times\left\lceil{\nicefrac{\sqrt{n}}{2\sigma}}\right\rceil}$ be the orthogonal projection matrix corresponding to ${Proj}_{\mathbf{1}_{J_2}}^{\perp}(\cdot)$. Write $P\Sigma_{J_2} P=AA^T$ for $A\in\R^{\left\lceil{\nicefrac{\sqrt{n}}{2\sigma}}\right\rceil\times \left(\left\lceil{\nicefrac{\sqrt{n}}{2\sigma}}\right\rceil-1\right)}$. Then $(X_{J_2}^Tz)\sim N(0,\|z\|^2_2\cdot \Sigma_{J_2})$ and so $\mathrm{Proj}_{\mathbf{1}_{J_2}}^{\perp}(X_{J_2}^Tz)\sim N(0,\|z\|^2_2\cdot P\Sigma_{J_2} P)$, and therefore $\mathrm{Proj}_{\mathbf{1}_{J_2}}^{\perp}(X_{J_2}^Tz)\stackrel{\mathcal{D}}{=}\|z\|^2_2\cdot Au$ for $u\sim N\left(0,\mathbf{I}_{\left\lceil{\nicefrac{\sqrt{n}}{2\sigma}}\right\rceil-1}\right)$. By examining the definition of $A$, we see that $u^T(A^TA)u\geq \|u\|^2_2\cdot \lambda^2_{\text{min}}\left(\Sigma_{\left[\left\lceil{\nicefrac{\sqrt{n}}{\s}}\right\rceil\right]}\right)$, therefore,
$$\|\mathrm{Proj}_{\mathbf{1}_{J_2}}^{\perp}(X_{J_2}^Tz)\|^2_2\stackrel{\mathcal{D}}{=}\|z\|^2_2\cdot \|Au\|^2_2\stackrel{\mathcal{D}}{\geq} 0.5n\cdot \lambda^2_{\text{min}}\left(\Sigma_{\left[\left\lceil{\nicefrac{\sqrt{n}}{\s}}\right\rceil\right]}\right)\cdot \chi^2_{\left\lceil{\nicefrac{\sqrt{n}}{2\sigma}}\right\rceil-1}\;.$$
 Furthermore, the number of such sets $J_2$ is bounded by $2^{\left\lceil{\nicefrac{\sqrt{n}}{\s}}\right\rceil}$. By the chi-square tail bounds from \citet{FD_NIPS}, using the assumption that $\nicefrac{\sqrt{n}}{\s}\geq 100$, we have
\begin{align*}
&Pr\left(\chi^2_{\left\lceil{\nicefrac{\sqrt{n}}{2\sigma}}\right\rceil-1}\leq \frac{\sqrt{n}}{100\s}\right)
\leq 
Pr\left(\chi^2_{\left\lceil{\nicefrac{\sqrt{n}}{2\sigma}}\right\rceil-1}\leq 0.02\cdot\tfrac{50}{49}\cdot \left(\left\lceil{\nicefrac{\sqrt{n}}{2\sigma}}\right\rceil-1\right)\right)\\
&\leq \exp\left\{\tfrac{1}{2}\left(\left\lceil{\nicefrac{\sqrt{n}}{2\sigma}}\right\rceil-2\right)\left(1-0.02\cdot\tfrac{50}{49}+\log\left(0.02\cdot\tfrac{50}{49}\right)\right)\right\}
\leq \exp\left\{\tfrac{1}{2}\left(\left\lceil{\nicefrac{\sqrt{n}}{2\sigma}}\right\rceil\cdot \tfrac{48}{50}\right)\left(1-0.02\cdot\tfrac{50}{49}+\log\left(0.02\cdot\tfrac{50}{49}\right)\right)\right\}\\
&\leq e^{-0.7084\left\lceil{\nicefrac{\sqrt{n}}{\sigma}}\right\rceil}\;.
\end{align*}
Therefore,
\begin{align*}
& Pr\left(\exists J_1\subset \left[\left\lceil{\sqrt{n}}\right\rceil\right], |J_1|\geq\tfrac{\sqrt{n}}{2}, \ \|\mathrm{Proj}_{\mathbf{1}_{J_1}}^{\perp}(X_{J_1}^Tz)\|^2_2\leq n\lambda^2_{\text{min}}\left(\Sigma_{\left[\left\lceil{\sqrt{n}}\right\rceil\right]}\right)\cdot \tfrac{\sqrt{n}}{200\s}\right)\\
 &\leq Pr\left(\exists J_2\subset [\left\lceil{\sqrt{n}}\right\rceil], |J_2|=\tfrac{\left\lceil{\sqrt{n}}\right\rceil}{2}, \ \|\mathrm{Proj}_{\mathbf{1}_{J_2}}^{\perp}(X_{J_2}^Tz)\|^2_2\leq n\lambda^2_{\text{min}}\left(\Sigma_{\left[\left\lceil{\sqrt{n}}\right\rceil\right]}\right)\cdot \tfrac{\sqrt{n}}{200\s}\right)\\
 &\leq 2^{\left\lceil{\nicefrac{\sqrt{n}}{\s}}\right\rceil}\cdot Pr\left(\chi^2_{\left\lceil{\nicefrac{\sqrt{n}}{2\sigma}}\right\rceil-1}\leq \tfrac{\sqrt{n}}{100\s}\right)\leq 2^{\left\lceil{\nicefrac{\sqrt{n}}{\s}}\right\rceil}\cdot e^{-0.7084\left\lceil{\nicefrac{\sqrt{n}}{\sigma}}\right\rceil}\leq e^{-0.015\left\lceil{\nicefrac{\sqrt{n}}{\sigma}}\right\rceil}\leq e^{-0.015\s^{-1}\sqrt{n}}\;.
\end{align*}

\bibliography{sparse_vector_rates}

\end{document}